\documentclass[reqno]{amsart}
\usepackage{hyperref}

\newtheorem{theorem}{Theorem}[section]

\newtheorem{lemma}[theorem]{Lemma}
\newtheorem{proposition}[theorem]{Proposition}
\newtheorem{corollary}[theorem]{Corollary}
\newtheorem{remark}[theorem]{Remark}

\newcommand{\R}{{\mathbb R}}
\newcommand{\N}{{\mathbb N}}

\newcommand{\C}{{\mathbb C}}

\newcommand{\nn}{\nonumber}
\newcommand{\be}{\begin{equation}}
\newcommand{\ee}{\end{equation}}
\newcommand{\bea}{\begin{eqnarray}}
\newcommand{\eea}{\end{eqnarray}}

\newcommand{\ol}{\overline}
\newcommand{\ti}{\tilde}

\newcommand{\E}{\mathrm{e}}
\newcommand{\I}{\mathrm{i}}

\newcommand{\im}{\mathrm{Im}}
\newcommand{\re}{\mathrm{Re}}

\newcommand{\floor}[1]{\lfloor#1 \rfloor}

\newcommand{\eps}{\varepsilon}

\newcommand{\sig}{\sigma}
\newcommand{\lam}{\lambda}

\numberwithin{equation}{section}

\begin{document}

\title[On the Absolutely Continuous Spectrum of SL Operators]{On the Absolutely Continuous
Spectrum of Sturm--Liouville Operators with Applications to Radial Quantum Trees}

\author[M. Schmied]{Michael Schmied}
\address{Faculty of Mathematics\\
Nordbergstrasse 15\\ 1090 Wien\\ Austria}

\author[R. Sims]{Robert Sims}
\address{Faculty of Mathematics\\
Nordbergstrasse 15\\ 1090 Wien\\ Austria}
\email{\href{mailto:Robert.Sims@univie.ac.at}{Robert.Sims@univie.ac.at}}

\author[G. Teschl]{Gerald Teschl}
\address{Faculty of Mathematics\\
Nordbergstrasse 15\\ 1090 Wien\\ Austria\\ and International Erwin Schr\"odinger
Institute for Mathematical Physics, Boltzmanngasse 9\\ 1090 Wien\\ Austria}
\email{\href{mailto:Gerald.Teschl@univie.ac.at}{Gerald.Teschl@univie.ac.at}}
\urladdr{\href{http://www.mat.univie.ac.at/~gerald/}{http://www.mat.univie.ac.at/\~{}gerald/}}

\thanks{{\it Oper. Matrices (to appear)}}
\thanks{{\it Research supported by the Austrian Science Fund (FWF) under Grant No.\ Y330}}


\begin{abstract}
We consider standard subordinacy theory for general Sturm--Liouville operators and
give criteria when boundedness of solutions implies that no subordinate solutions exist.
As applications, we prove a Weidmann-type result for general Sturm--Liouville operators
and investigate the absolutely continuous spectrum of radially symmetric quantum trees.
\end{abstract}

\maketitle

%
%
%

\section{Introduction}

Schr\"odinger operators on graphs (both discrete and continuous) have a long tradition
in both the physics and mathematics literature. In particular, the literature on this subject
is quite extensive and we only refer to \cite{car1,car2,ku1,ku2} and the references therein as
a starting point. For related results on random trees we refer to \cite{asw1,asw2}.
The purpose of this note is to investigate the relation of the growth of solutions
with the spectral properties for such quantum graphs. In \cite{ku2} it was shown that
if the graph is of sub-exponential grow, then existence of a bounded solution implies that
the corresponding energy is in the spectrum. Moreover, for Schr\"odinger operators on the
line it is well-known (\cite{gilper,si,Stolz2}) that boundedness of solutions implies that the corresponding
energy is in the absolutely continuous spectrum. Our motivation was to prove such kind of results
for quantum trees. As a first step we consider radially symmetric quantum trees, that is, trees whose
branching and edge lengths depend only on the distance from the root, which can be reduced
to the study of general Sturm--Liouville equations with weights (\cite{car2,so1}).

Hence the purpose of this paper is twofold, to relate boundedness of solutions of general Sturm--Liouville
equations with the existence of purely absolutely continuous spectrum and to apply these results to 
radially symmetric quantum trees.

We begin by fixing our notation.
We will consider Sturm--Liouville operators on $L^{2}((a,b), r\,dx)$
with $-\infty \le a<b \le \infty$ of the form
\begin{equation} \label{stli}
\tau = \frac{1}{r} \Big(- \frac{d}{dx} p \frac{d}{dx} + q \Big),
\end{equation}
where the coefficients $p,q,r$ are real-valued
satisfying
\begin{equation} \label{eq:baseass}
p^{-1},q,r \in L^1_{loc}(a,b), \quad p,r>0.
\end{equation}
We will use $\tau$ to describe the formal differentiation expression and
$H$ for the operator given by $\tau$ with separated boundary conditions at
$a$ and/or $b$.

If $a$ (resp.\ $b$) is finite and $q,p^{-1},r$ are in addition integrable
near $a$ (resp.\ $b$), we will say $a$ (resp.\ $b$) is a \textit{regular}
endpoint.  We will say $\tau$ respectively $H$ is \textit{regular} if
both $a$ and $b$ are regular. 

For every $z\in\C\backslash\sig_{ess}(H)$ there is a unique (up to a constant) solution
$u_a(z,x)$ of $\tau u = z u$ which is in $L^2$ near $a$ and satisfies the
boundary condition at $a$ (if any). Similarly there is such a solution $u_b(z,x)$
near $b$.

For the purpose of investigating the absolutely continuous spectrum it is well-known
that it suffices to consider the case where one endpoint, say $a$, is regular.
In this case a key role is played the Weyl $m$-function
\be
m_b(z) = \frac{pu_b'(z,a)}{u_b(z,a)}
\ee
which is a Herglotz function and satisfies
\be\label{immb}
\im( m_b(z) ) =  \im(z) \int_a^b |u_b(z,x)|^2 r(x)\, dx,
\ee
if $u_b(z,x)$ is normalized according to
\be \label{ubweylm}
u_b(z,x) = c(z,x) + m_b(z) s(z,x).
\ee
Here $c(z,x)$ and $s(z,x)$ are the solutions of $\tau u = z u$ corresponding to the
initial conditions $c(z,a)= ps'(z,a)=1$, $s(z,a)= pc'(z,a)=0$.

In addition, we will also need the Weyl $m$-functions $m_{b,\alpha}(z)$ corresponding to
the boundary conditions
\be \label{eqbc}
\cos(\alpha) f(a) - \sin(\alpha) pf'(a) =0.
\ee
Then $m_b(z)=m_{b,0}(z)$ corresponds to the Dirichlet boundary condition $f(a)=0$ and
we have the following well-known relation
\be \label{eqwmab}
m_{b,\alpha}(\lam) = \frac{\cos(\alpha-\beta) m_{b,\beta}(\lam) + \sin(\alpha-\beta)}{\cos(\alpha-\beta)
- \sin(\alpha-\beta) m_{b,\beta}(\lam)}.
\ee
We refer the interested reader to \cite{CodLev,te,wdln} for relevant background information.

%
%
%
%

\section{Subordinacy}

In this section, we will present a streamlined approach to those aspects of 
the method of subordinacy, as introduced by Gilbert and Pearson in \cite{gilper}
(see \cite{clhi} for the case of Sturm--Liouville operators, see also \cite{cl}), 
which pertain to the absolutely continuous spectrum of the operator $H$ defined
in the previous section.  As our applications involve Sturm--Liouville equations, we
will discuss this method in precisely that context. For a more elementary approach
under somewhat stronger assumptions (a kind of uniform subordinacy) we refer to
Weidmann \cite{wd2}. Without loss of generality, we will assume that $a$ is regular 
and $b$ is limit point.

The following facts are well-known: the self-adjoint operator $H$ corresponding to (\ref{stli}), with
a suitable choice of boundary condition, is unitarily equivalent to multiplication by $\lam$ in the space 
$L^2(\R,d\mu)$, where $\mu$ is the measure associated to the Weyl $m$-function. 
For this reason, the set
\be
M_s \, = \, \{\lam \in \mathbb{R} \, | \,  \limsup_{\eps\downarrow 0}\im(m_b(\lam+\I\eps)) \, = \, \infty \, \}
\ee
is a support for the singularly continuous spectrum of $H$, written as $\sigma_{ {\rm sc}}(H)$, 
and moreover,
\be
M_{ac} = \{\lam \in \mathbb{R} \, | \, 0\, < \, \limsup_{\eps\downarrow 0} \im(m_b(\lam+\I\eps)) \, < \, \infty \, \}
\ee
is a minimal support for the absolutely continuous spectrum, similarly written as $\sigma_{{\rm ac}}(H)$.

One also has that $\sig_{{\rm ac}}(H)$ can be recovered from the essential
closure of $M_{ac}$, that is,
\begin{equation}
\sig_{{\rm ac}}(H) \, = \, \ol{M}_{ac}^{ess} \, = \, \{ \lam\in\R \, | \, | \, (\lam-\eps,\lam+\eps) \cap M_{ac} \, | \, > \, 0
\mbox{ for all } \eps \, > \, 0\},
\end{equation}
where $|A|$ denotes the Lebesgue measure of the set $A \subset \mathbb{R}$.

Before we begin our discussion of subordinacy, we present a crucial estimate on 
the imaginary part of the Weyl $m$-function. Let
\begin{equation}
\| f \|_{(a,x)} = \sqrt{\int_a^x |f(y)|^2 r(y) dy}, \qquad x\in(a,b),
\end{equation}
denote the norm of $f \in L^2((a,x), r dy)$. Also, for fixed $\lambda \in \mathbb{R}$, let
$s(\lambda,x)$ (resp. $c(\lambda,x)$) denote the solution of $(\tau - \lambda)u=0$ satisfying
a Dirichlet (resp. Neumann) boundary condition at the regular endpoint $a$. Define $\eps : (a,b) \to (0, \infty)$
by setting 
\begin{equation} \label{defeps}
\eps \, = \, \eps_{\lambda}(x) \, = \, \left( 2 \, \| s(\lambda) \|_{(a,x)} \, \| c(\lambda) \|_{(a,x)} \right)^{-1}.
\end{equation}
As indicated by the notation above, $\eps$ depends on both $\lambda$ and $x$, but we will often suppress this
in our notation below. Observe that for $\lambda \in \mathbb{R}$ fixed, the assumption that $b$ is limit point
guarantees that there is a one-to-one correspondence between
$\eps \in (0, \infty)$ and $x \in (a,b)$. The following estimate was proven by Jitomirskaya and Last:

\begin{lemma}[\cite{jl}] \label{lemjlineq}
Fix $\lambda \in \mathbb{R}$ and define $\eps$ as in (\ref{defeps}) above.
The estimate
\begin{equation}
5 -\sqrt{24} \le | m_b(\lam+\I \eps)| \frac{\| s(\lam) \|_{(a,x)}}{\| c(\lam) \|_{(a,x)}}
\le 5 +\sqrt{24},
\end{equation}
is valid.
\end{lemma}
We present a proof for the sake of completeness.
\begin{proof}
Let $x \geq t \geq a$. By variation of constants, the solution $u_b(\lambda + i \eps)$, as
defined in (\ref{ubweylm}), can be written
as
\be 
\begin{split}
u_b(\lam+\I\eps)(t) = c(\lam,t) &- m_b(\lam+\I\eps) s(\lam)(t)\\
& - \I\eps\int_a^t \big( c(\lam,t) s(\lam,y) -
c(\lam,y) s(\lam,t) \big) u_b(\lam+\I\eps,y) \, r(y) dy.
\end{split}
\ee
Hence one obtains after a little calculation
\bea \nn
&& \|c(\lam) - m_b(\lam+\I\eps) s(\lam) \|_{(a,x)} \le \|u_b(\lam+\I\eps) \|_{(a,x)}\\
&& \qquad {} +2 \eps \| s(\lam) \|_{(a,x)} \| c(\lam) \|_{(a,x)}
\|u_b(\lam+\I\eps) \|_{(a,x)}.
\eea
Using the definition of $\eps$ and (\ref{immb}) we obtain
\bea \nn
&& \|c(\lam) - m_b(\lam+\I\eps) s(\lam) \|_{(a,x)}^2 \le
4 \|u_b(\lam+\I\eps) \|_{(a,x)}^2\\ \nn
&& \qquad \le 4 \|u_b(\lam+\I\eps) \|_{(a,b)}^2 = \frac{4}{\eps} \im(m_b(\lam+\I\eps))\\
&& \qquad \le 8 \| s(\lam) \|_{(a,x)} \|c(\lam) \|_{(a,x)}
\im(m_b(\lam+\I\eps)).
\eea
Combining this estimate with
\begin{equation}
\|c(\lam) - m_b(\lam+\I\eps) s(\lam) \|_{(a,x)}^2 \ge
\Big(\|c(\lam)\|_{(a,x)} - |m_b(\lam+\I\eps)| \| s(\lam) \|_{(a,x)}\Big)^2
\end{equation}
shows $(1-t)^2 \le 8 t$, where $t= |m_b(\lam+\I \eps)| \| s(\lam) \|_{(a,x)}
\| c(\lam) \|_{(a,x)}^{-1}$.
\end{proof}

We now introduce the concept of subordinacy.
A nonzero solution $u$ of $\tau u = z u$ is called subordinate
at $b$ with respect to another solution $v$ if
\begin{equation} \label{defsubos}
\lim_{x\to b} \frac{\| u \|_{(a,x)}}{\| v \|_{(a,x)}} =0.
\end{equation}
It is easy to see that if $u$ is subordinate
with respect to $v$, then it is subordinate with respect to any linearly
independent solution. In particular, a subordinate solution is unique up
to a constant. Moreover, if a solution $u(\lam)$ of $\tau u = \lam u$, $\lam\in\R$,
is subordinate, then it is real up to a constant, since both the real and
the imaginary part are subordinate. For $z\in\C\backslash\R$ we know that there
is always a subordinate solution at $b$, namely $u_b(z,x)$.
The following result considers the case $z\in\R$.

\begin{lemma} \label{lemsuiffmpm}
Let $\lam\in\R$. There is a solution $u$ of $\tau u = \lam u$ that is subordinate at $b$  if
and only if $m_b(\lam+\I\eps)$ converges to a limit in $\R\cup\{\infty\}$ as
$\eps\downarrow 0$. Moreover,
\begin{equation}
\lim_{\eps\downarrow 0} m_b(\lam+\I\eps) =  \frac{\cos(\alpha) pu'(\lam,a) +
\sin(\alpha) u(\lam,a)}{\cos(\alpha) u(\lam,a) - \sin(\alpha) pu'(\lam,a)}
\end{equation}
in this case.
\end{lemma}

\begin{proof}
We will consider the number $\alpha$ fixing the boundary condition (\ref{eqbc}) as a parameter.
Denote by $s_\alpha(z,x)$, $c_\alpha(z,x)$ the solutions of $\tau u = z u$ corresponding to the
initial conditions $s_\alpha(z,a)= - \sin(\alpha)$, $ps_\alpha'(z,a) = \cos(\alpha)$,
$c_\alpha(z,a)= \cos(\alpha)$, $pc_\alpha'(z,a) = \sin(\alpha)$.

Let $L_\alpha$ be set of all $\lam\in\R$ for which the limit $\lim_{\eps\downarrow0} m_{b,\alpha}(\lam+\I\eps)$
exists (finite or infinite). Then (\ref{eqwmab}) implies
that $L_\alpha=L_\beta$. Hence $L\equiv L_\alpha$ is independent of $\alpha$. We set
$m_{b,\alpha}(\lam)=\lim_{\eps\downarrow0} m_{b,\alpha}(\lam+\I\eps)$ for $\lam\in L$.

Moreover, every solution can (up to a constant) be written as $s_\beta(\lam,x)$ for some
$\beta\in[0,\pi)$. But by Lemma~\ref{lemjlineq} $s_\beta(\lam,x)$ is subordinate if and only
if $\lim_{\eps\downarrow 0} m_{b,\beta}(\lam+\I\eps) =\infty$ and this is the case if and
only if
$$
m_{b,\alpha}(\lam) = \frac{\cos(\alpha-\beta) m_{b,\beta}(\lam) + \sin(\alpha-\beta)}{\cos(\alpha-\beta)
- \sin(\alpha-\beta) m_{b,\beta}(\lam)} = -\cot(\alpha-\beta)
$$
is a number in $\R\cup\{\infty\}$.
\end{proof}

We are interested in $N(\tau)$, the set of all $\lam\in\R$ for which no subordinate solution exists, that is,
\begin{equation} \label{defnspm}
N(\tau)= \{\lam\in\R | \mbox{No solution of $\tau u =\lam u$ is subordinate at $b$} \}.
\end{equation}

\begin{remark}
Since the set, for which the limit $\lim_{\eps\downarrow 0} m_b(\lam+\I\eps)$
does not exist (finite or infinite), is of zero spectral and Lebesgue measure,
changing the $\lim$ in (\ref{defsubos}) to a $\liminf$ will affect $N(\tau)$ only
on such a set (which is irrelevant for our purpose).
\end{remark}

Then, as consequence of the previous lemma, we have

\begin{theorem} \label{thmnpmacsp}
The set $N(\tau)$ is a minimal support for the absolutely continuous spectrum of $H$.
In particular,
\begin{equation}
\sig_{ac}(H)= \ol{N(\tau)}^{ess}.
\end{equation}
Moreover, the set
\be
\{ \lam |\, s(\lam,x) \text{ is subordinate at $b$} \}
\ee
is a minimal support for the singular spectrum.
\end{theorem}

\begin{proof}
Without loss of generality we may assume $m_b(\lam+\I 0)$ exists (finite or infinite).
But for those values of $\lam$ the cases $\im( m_b(\lam+\I 0) )\in \{0,\infty\}$ imply
$\lam\not\in N(\tau)$ by Lemma~\ref{lemsuiffmpm}. Thus
we have $0 < \im( m_b(\lam+\I 0) ) <\infty$ if and only if
$\lam\in N(\tau)$ and the first result follows since $M_{ac}$ is a minimal support for the
absolutely continuous spectrum.

On the other hand, $s(\lam,x)$ is subordinate if and only if $m_b(\lam+\I 0)=\infty$ and
the second result follows since $M_s$ is a minimal support for the singular spectrum.
\end{proof}

Note that if $(\lam_1,\lam_2) \subseteq N(\tau)$, then the spectrum of any self-adjoint
extension $H$ of $\tau$ is purely absolutely continuous in the interval $(\lam_1,\lam_2)$.

\begin{remark}
As in \cite{jl} one can also give supports for the $\alpha$ continuous spectrum of $H$,
that is, the part which is absolutely continuous with respect to $\alpha$-dimensional
Hausdorff measure.
\end{remark}

We will now prove a simple lemma which enables one to show the lack of subordinate solutions
by verifying certain solution estimates.

\begin{lemma} \label{lem:nosub1}
Let the coefficients of $\tau$ satisfy the basic assumptions given above, i.e. (\ref{eq:baseass}). 
Suppose that  for any solution $u$ of $\tau u = \lam u$ there exists a constant $C = C(u)$ for which
\be\label{eqrubup}
\limsup_{x\to b} \frac{1}{x} \int_a^x \frac{|pu'(y)|^2}{r(y)} \, dy \le C
\limsup_{x\to b} \frac{1}{x} \int_a^x |u(y)|^2 r(y) dy.
\ee
If, in addition, $x^{-1}\|u\|_{(a,x)}^2$ is bounded for every solution $u$ of $\tau u = \lam u$,
then $\lam \in N(\tau)$.
\end{lemma}

\begin{proof}
Without loss of generality, we prove this result in the case that $a=0$. 
Suppose, under the assumptions above, there were a subordinate solution $u$, and let $v$
be a second linearly independent solution with Wronskian
\begin{equation}
W[u_1,v_1](y) = u_1(y)pv_1'(y) - pu_1'(y) v_1(y) =1
\end{equation}
for all $y \in (0,b)$. Since $u$ is subordinate we know
$$
\lim_{x\to b} \frac{x^{-1}\|u\|_{(0,x)}^2}{x^{-1}\|v\|_{(0,x)}^2} =0
$$
and boundedness of $x^{-1}\|v\|_{(0,x)}^2$ even implies $x^{-1}\|u\|_{(0,x)}^2\to 0$. Moreover,
by (\ref{eqrubup}) we also have $x^{-1}\|r^{-1} pu'\|_{(0,x)}^2\to 0$. But then
$$
1 = \frac{1}{x} \int_0^x W[u,v](y) dy \le \frac{1}{x}\|u\|_{(0,x)} \|r^{-1} pv'\|_{(0,x)}
+ \frac{1}{x} \|r^{-1} pu'\|_{(0,x)} \|v\|_{(0,x)} \, \to 0
$$
gives the desired contradiction.
\end{proof}

The next result, which is a generalization of Simon \cite[Lem.~3.1]{si}, demonstrates an explicit estimate
on the derivative of a solution to $\tau u = \lambda u$ in terms of the local $L^2$ norm of that
solution. As a consequence, see Corollary~\ref{cor:derbd}, we are able to provide conditions
on the coefficients of $\tau$ which allow one to verify the assumption (\ref{eqrubup}) of
Lemma~\ref{lem:nosub1}.

For any interval $[x-1,x+1] \subset (a,b)$, define the quantities,
\begin{equation} \label{eq:P}
P(x) = \int_{x - 1/2}^{x+1/2} \frac{1}{p(y)} dy,
\end{equation}
and 
\begin{equation} \label{eq:rpm}
r_-(x) \, = \, \inf_{y \in [x-1,x+1]} r(y) \quad {\rm and} \quad  r_+(x) = \sup_{y \in [x-1,x+1]} r(y).
\end{equation}
We will assume that for each such $x$, we have $0 < r_-(x) \leq r_+(x) < \infty$, and
moreover, we set
\begin{equation} \label{eq:gammy}
\gamma(x) = \frac{r_+(x)}{r_-(x)}.
\end{equation} 

\begin{lemma} \label{lem:derbd}
Let the coefficients of $\tau$ satisfy the general assumptions (\ref{eq:baseass}), and 
let $u$ be a real-valued solution of $\tau u = \lam u$ on (a,b).  
For any interval $[x-1,x+1] \subset (a,b)$, suppose that $0 < r_-(x) \leq r_+(x) < \infty$,
where $r_{\pm}(x)$ are as defined in (\ref{eq:rpm}) above. Then, the bound
\begin{equation} \label{eq:squaredbd}
\frac{(pu')(x)^2}{r(x)} \, \leq \,  \, \gamma(x) \, \int_{x-1}^{x+1} \left[ \, \frac{2}{P(x) \, r(y)} \, + \, \left| \, \frac{q(y)}{r(y)} \, - \, \lambda \, \right| \,\right]^2 \, dy \, \cdot \, \int_{x-1}^{x+1} u^2(y) \, r(y) \, dy. 
\end{equation}
holds, where $P(x)$ and $\gamma(x)$ are defined by (\ref{eq:P}) and (\ref{eq:gammy}), respectively.
\end{lemma}

\begin{proof}
Without loss of generality, we take $a=-1$, and prove this result first at $x=0$. The general result
claimed in (\ref{eq:squaredbd}) will follow by translation.

To see this estimate for $x=0$, we introduce the function $P_0: (a,b) \to \mathbb{R}$ by setting
\begin{equation*} \label{eq:I}
P_0(x) \, = \, \int_0^x \frac{1}{p(t)} \, dt .
\end{equation*}
Let $f$ be any solution of $\tau u = \lam u$ on $[-1,1]$ and take any number $0 \leq x \leq 1$. Integration by parts yields
\begin{equation*}
\int_0^x (pf')'(y) \, \left[ P_0(x) \, - \, P_0(y) \right] \, dy \, = \, f(x) \, - \, f(0) \, - \, (pf')(0) \, P_0(x),  
\end{equation*} 
and similarly,
\begin{equation*}
\int_{-x}^0 (pf')'(y) \, \left[ P_0(-x) \, - \, P_0(y) \right] \, dy \, = \, f(0) \, - \, f(-x) \, + \,  (pf')(0) \, P_0(-x).  
\end{equation*} 
Rewriting things a bit, we find that
\begin{equation*}
\begin{split}
(pf')(0) \, \left[ P_0(x) \, - \, P_0(-x) \right] \, = \, f(x) \, &- \, f(-x) \, - \, \int_0^x (pf')'(y) \, \left[ P_0(x) \, - \, P_0(y) \right] \, dy \\
& + \, \int_{-x}^0 (pf')'(y) \, \left[ P_0(y) \, - \, P_0(-x) \right] \, dy,
\end{split}
\end{equation*}
and since $P_0(x) -P_0(y) \leq P_0(x) -P_0(-x)$ for every $0 \leq y \leq x$ (and similarly  $P_0(y) -P_0(-x) \leq P_0(x) -P_0(-x)$ for every $-x \leq y \leq 0$), it
is easy to see then 
\begin{equation*}
|(pf')(0)| \, \leq \, \frac{|f(x) - f(-x)|}{P_0(x)-P_0(-x)} \, + \, \int_{-x}^x \left| (pf')'(y) \right| \, dy.
\end{equation*}
Integrating the above from $1/2$ to $1$, we find the bound
\begin{equation} \label{eq:intbd}
|(pf')(0)| \, \leq \,  \frac{2}{P(0)} \int_{-1}^1 |f(y)| \, dy \, +  \, \int_{-1}^1\left| (pf')'(y) \right| \, dy,
\end{equation}
where $P(0) = P_0(1/2) -P_0(-1/2)$ as defined in (\ref{eq:P}) above.
Multiplying both sides of (\ref{eq:intbd}) by $r_+(0)^{-1/2}$, it is clear that
\begin{equation}
\begin{split}
\sqrt{ \gamma^{-1}(0)} \cdot \frac{|(pf')(0)|}{\sqrt{r(0)}} & \leq \, \frac{|(pf')(0)|}{\sqrt{r_+(0)}} \\
& \leq \,  \frac{2}{P(0)} \int_{-1}^1 \frac{|f(y)|}{\sqrt{r(y)}} \, dy 
 \, + \,  \int_{-1}^1\frac{| (pf')'(y) |}{\sqrt{r(y)}} \, dy.
\end{split}
\end{equation}

If we set $f=u$, where we now regard $u$ as a solution only over $[-1,1]$, we find that
\begin{equation}
\frac{|(pu')(0)|}{ \sqrt{r(0)}} \, \leq \,  \, \sqrt{\gamma(0)} \, \int_{-1}^1 \left[  \, \frac{2}{P(0) \, r(y)} \, + \, \left| \, \frac{q(y)}{r(y)} \, - \, \lambda \, \right| \, \right] \, |u(y)| \, \sqrt{r(y)} \, dy.
\end{equation}
An application of H\"older yields,
\begin{equation}
\frac{(pu')(0)^2}{r(0)} \, \leq \,  \, \gamma(0) \, \int_{-1}^1 \left[ \, \frac{2}{P(0) \, r(y)} \, + \, \left| \, \frac{q(y)}{r(y)} \, - \, \lambda \, \right| \,\right]^2 \, dy \, \cdot \, \int_{-1}^1 u^2(y) \, r(y) \, dy. 
\end{equation}
The result claimed in Lemma~\ref{lem:derbd} now follows by translation.
If $[x-1,x+1] \subset (a,b)$, take $f: [-1,1] \to \mathbb{R}$ to be given by 
$f(y) = u(x+y)$. (\ref{eq:squaredbd}) follows after similarly translating
the coefficients $p$, $q$, and $r$.
\end{proof}

\begin{corollary} \label{cor:derbd}
Let $(a,b) =(0, \infty)$. If the functions
\be \label{eq:boundme}
\gamma(x), \quad \frac{1}{P(x)^2} \, \int_{x-1}^{x+1} \frac{1}{r(y)^2} \, dy, \quad \mbox{and} \quad \int_{x-1}^{x+1} \left| \frac{q(y)}{r(y)}-\lam \right|^2 \,dy
\ee
are all bounded for $x>1$, then equation (\ref{eqrubup}) holds. In this case, if all solutions of $\tau u = \lambda u$ are bounded,
in the sense that $\sqrt{r} u$ is a bounded function, then $\lambda \in N(\tau)$. 
\end{corollary}

\begin{proof}
Clearly, if $\sqrt{r} u$ is a bounded function, then $ x^{-1} \| u \|_{(0,x)}$ is bounded. Thus, Lemma~\ref{lem:nosub1} can
be applied if we verify (\ref{eqrubup}).

Let $\gamma_+$, $P_+$ and $Q_+$ be the bounds on the quantities listed in (\ref{eq:boundme}) respectively.
It is easy to see that for any $t>1$,
\begin{equation} \label{eq:simpbd}
\gamma(t) \cdot \int_{t-1}^{t+1} \left[ \frac{1}{P(t)r(s)} + \left| \frac{q(s)}{r(s)} - \lambda \right| \right]^2 \, ds 
\leq 2 \gamma_+ (P_+ + Q_+).
\end{equation}
Therefore, for all $x>1$,
\begin{eqnarray}
\int_0^x \frac{|pu'(t)|^2}{r(t)} \, dt & = & \int_0^1 \frac{|pu'(t)|^2}{r(t)} \, dt + \int_1^x \frac{|pu'(t)|^2}{r(t)} \, dt \nonumber \\
& \leq & \int_0^1 \frac{|pu'(t)|^2}{r(t)} \, dt + 2 \gamma_+(P_+ + Q_+) \int_1^x  \int_{t-1}^{t+1} |u(s)|^2 r(s) \, ds \, dt,
\end{eqnarray}
where we used both Lemma~\ref{lem:derbd} and the bound (\ref{eq:simpbd}). Changing the order of integration, we see that
\begin{equation}
\int_1^x \int_{t-1}^{t+1}  |u(s)|^2 r(s) ds dt \, \leq \, 2 \,  \int_1^{x+1} |u(s)|^2 r(s) ds,
\end{equation}
and hence, (\ref{eqrubup}) holds.
\end{proof}

In order to prove there is no subordinate solution of $\tau u = \lam u$, Lemma~\ref{lem:nosub1} 
requires estimates of the derivative, i.e.  $\| r^{-1} pu' \|_{(a,x)}$, explicitly in terms of the 
solution $\| u \|_{(a,x)}$. If the potential $q$ satisfies a locally uniform $L^2$ estimate, then 
Lemma~\ref{lem:derbd} provides the desired bounds.  

If one weakens the assumptions on the potential $q$, then one may still prove that
bounded solutions imply no subordinate solutions. In this case, however, explicit derivative bounds 
are not readily available, and so Lemma~\ref{lem:nosub1} does not immediately apply.
Let $q$ be written in terms of it's positive and negative parts, i.e. as $q= q_+ - q_-$.
For $q$, we will assume that
\begin{equation} \label{eq:qass}
q_+ \in L^1_{{\rm loc}}((a,b)) \quad {\rm and} \quad \sup_{x \in (a,b)} \int_x^{x+1} \frac{q_-(t)}{r(t)} dt \, < \, \infty.
\end{equation}
In terms of the coefficient $p$, we will suppose that
\begin{equation} \label{eq:I-}
\inf_{x \in (a,b)} \int_x^{x+1} \frac{r(t)}{p(t)} dt  \, = \, I_- \, > \, 0.
\end{equation}
With $r$ we assume
\begin{equation} \label{eq:gamfine}
\gamma \, = \, \sup_{x \in (a,b)} \frac{r_+(x)}{r_-(x)} \, < \, \infty,
\end{equation}
where $r_{\pm}(x)$ are as defined in (\ref{eq:rpm}).

We will prove the following result.

\begin{lemma} \label{lem:bdnosub} 
Let the coefficients of $\tau$ satisfy the general assumptions of (\ref{eq:baseass}). Moreover,
suppose that $r$ is non-decreasing and each of (\ref{eq:qass}), (\ref{eq:I-}), and (\ref{eq:gamfine}) hold.  
In this case, if all solutions of $\tau u =  \lambda u$ are such that $\sqrt{r} u$ is
bounded on $(a,b)$, then $\lambda \in N(\tau)$.
\end{lemma}

In the Schr\"odinger setting, i.e. in the case $p = r =1$, such a lemma was proven by Stolz in \cite[Lemma 4]{Stolz2}.
Our proof follows closely his ideas, but generalizes the setting to the context of Sturm--Liouville equations.
Lemma~\ref{lem:bdnosub} is an immediate consequence of  the following two propositions. 

\begin{proposition} \label{prop:bdubdu'}
Let the coefficients of $\tau$ satisfy the general assumptions of (\ref{eq:baseass}). Moreover,
suppose that $r$ is non-decreasing and each of (\ref{eq:qass}), (\ref{eq:I-}), and (\ref{eq:gamfine}) hold. 
Fix $\lambda \in \mathbb{R}$, and let $u$ be a solution of $\tau u = \lambda u$ on $(a,b)$.
If $\sqrt{r} u$ is bounded on $(a,b)$, then $\sqrt{r}^{-1}pu'$ is also bounded on $(a,b)$.
\end{proposition}

\begin{proposition} \label{prop:l2bd}
Let the coefficients of $\tau$ satisfy the general assumptions of (\ref{eq:baseass}). Moreover,
suppose that both (\ref{eq:I-}) and (\ref{eq:gamfine}) hold. 
Fix $\lambda \in \mathbb{R}$, and let $u$ be a solution of $\tau u = \lambda u$ on $(a,b)$ for which
there exist constants $C_1 =C_1(u)$ and $C_2 =C_2(u)$ such that
\begin{equation} \label{eq:uu'bd}
0 < C_1 \leq \sqrt{r(x)} |u(x)| + \frac{|pu'(x)|}{\sqrt{r(x)}} \leq C_2 < \infty,
\end{equation}
for all $x \in (a,b)$. Then, there exists a constant $C_3 = C_3(u)$ and
an $x_0 \in (a,b)$ for which 
\begin{equation} \label{eq:l2bd}
\int_a^x |u(t)|^2 r(t) dt \geq C_3 |x -a|,
\end{equation}
for all $b \geq x \geq x_0.$
\end{proposition}

Proposition~\ref{prop:bdubdu'} proves that if a solution is bounded, then so is it's derivative. There is no explicit estimate of the derivative in terms
of the bounded solution, however. Proposition~\ref{prop:l2bd} demonstrates that if the sum of the solution and it's 
derivative are bounded from above and below, then the $L^2$-norm of the solution grows linearly.

Given Proposition~\ref{prop:bdubdu'} and Proposition~\ref{prop:l2bd}, it is easy to prove Lemma~\ref{lem:bdnosub}

\begin{proof}(of Lemma~\ref{lem:bdnosub}) 
To prove that there are no subordinate solutions, we will show that for any two solutions
$u_1$ and $u_2$ of $\tau u  = \lambda u$ there exists a constant $C = C(u_1,u_2)$ for which
\begin{equation}
\lim_{x \to b} \frac{ \int_a^x |u_1(t)|^2 r(t) dt }{ \int_a^x|u_2(t)|^2 r(t)dt}  \geq C > 0.
\end{equation}
By assumption, $\sqrt{r}u_2$ is bounded on $(a,b)$, and therefore it is clear that
\begin{equation}
\int_a^x |u_2(t)|^2 r(t) dt \leq C(u_2) |x-a|,
\end{equation}
for all $x \in (a,b)$. We need only find a matching lower bound for $u_1$.

We have assumed that for all solutions $u$ of $\tau u  = \lambda u$, the function $\sqrt{r}u$ is bounded. 
By Proposition~\ref{prop:bdubdu'}, the same is then true for $\sqrt{r}^{-1}pu'$. Thus, for $u_1$ the upper bound
\begin{equation}
\sqrt{r(x)} |u_1(x)| + \frac{|pu_1'(x)|}{\sqrt{r(x)}} \leq C(u_1),
\end{equation}
follows from Proposition~\ref{prop:bdubdu'}. It is also easy to derive a lower bound. Let
$v_1$ be the solution of $\tau u = \lambda u$ for which the Wronskian of $u_1$ and $v_1$ is $1$.
Then for all $x \in (a,b)$,
\begin{equation}
1 = W[u_1,v_1](x) = u_1(x)pv_1'(x) - pu_1'(x) v_1(x),
\end{equation}
and thus
\begin{eqnarray}
1 & \leq & \sqrt{r(x)}|u_1(x)| \cdot \frac{|pv_1'(x)|}{\sqrt{r(x)}} + \frac{|pu_1'(x)|}{\sqrt{r(x)}} \cdot \sqrt{r(x)}|v_1(x)| \nonumber \\
& \leq & C_2(v_1) \, \left( \sqrt{r(x)} |u_1(x)| + \frac{|pu_1'(x)|}{\sqrt{r(x)}} \right),
\end{eqnarray}
yields the desired, pointwise, lower bound. $C_2(v_1)$ is just the maximum of the bound on $\sqrt{r} |v_1|$ and $\sqrt{r}^{-1} |pv_1'|$.
The lower bound on the $L^2$ norm now follows from Proposition~\ref{prop:l2bd}.
We have proven Lemma~\ref{lem:bdnosub}
\end{proof}

We remark that, as was pointed out in \cite{Stolz2}, the above proof of Lemma~\ref{lem:bdnosub} demonstrates that boundedness of solutions implies that
$\tau$ is limit point at $b$; since the norm of all solutions grows linearly.

\begin{proof}(of Proposition~\ref{prop:bdubdu'}) 
Since the coefficients (and $\lambda$) are real, it is sufficient to prove this result for real valued solutions $u$.
It is also technically advantageous to extend the solution $u$ considered in the statement of Proposition~\ref{prop:bdubdu'}
to a function which is bounded on $\mathbb{R}$. To do so, we extend the differential expression $\tau$ to $\tilde{\tau}$
defined on all of $\mathbb{R}$ by
setting $r = p = 1$ for all $x \in \mathbb{R} \setminus (a,b)$ and
for such values of $x$ take $q(x)=q_0$,  a constant for which $\lambda - q_0 >0$. 
In this case, the continuation of $u$
beyond $(a,b)$, i.e. the solution $\tilde{u}$ of $\tilde{\tau} u =  \lambda u$ which equals $u$ on $(a,b)$, is bounded; as is it's derivative.
 
Denote by $S_+ = \{ x \in \mathbb{R} : u(x) \geq 0 \}$ and similarly $S_- = \{ x \in \mathbb{R}: u(x) \leq 0\}$.  
We will prove that $\sqrt{r}^{-1}pu'$ is bounded on $S_+$; boundedness on $S_-$ will then follow by considering the
solution $v = -u$.

The proof that $\sqrt{r}^{-1}pu'$ is bounded on $S_+$ goes via contradiction. If $\sqrt{r}^{-1}pu'$ is unbounded on $S_+$, then
either:
\newline i) For every $n \in \mathbb{N}$, there exists $\xi_n \in S_+$ with $pu'(\xi_n) \geq n \sqrt{r(\xi_n)}$,
\newline \indent or
\newline ii) For every $n \in \mathbb{N}$, there exists $\xi_n \in S_+$ with $pu'(\xi_n) \leq -n \sqrt{r(\xi_n)}$.

By reflection, i.e. considering $v(x) = u(-x)$ for similarly reflected coefficients, the case of ii) will
follow from i). Assume i).

Let $\xi_n$ be as in the statement of i) and take $x$ such that $[ \xi_n, x] \subset S_+$.
It is easy to see then that
\begin{eqnarray}
u(x) - u(\xi_n) & = & \int_{\xi_n}^x \frac{r(t)}{p(t)} \, \frac{pu'(t)}{\sqrt{r(t)}} \, \frac{1}{\sqrt{r(t)}} \, dt \nonumber \\
& = &  \int_{\xi_n}^x \frac{r(t)}{p(t)} \, \left[ \frac{pu'(\xi_n) \, + \, \int_{\xi_n}^t (pu')'(s) \, ds}{\sqrt{r(t)}} \right] \, \frac{1}{\sqrt{r(t)}} \, dt  \nonumber \\
& = &  \int_{\xi_n}^x \frac{r(t)}{p(t)} \, \left[ \frac{pu'(\xi_n) \, + \, \int_{\xi_n}^t \left(q(s) - \lambda r(s) \right) u(s) \, ds}{\sqrt{r(t)}} \right] \, \frac{1}{\sqrt{r(t)}} \, dt. 
\end{eqnarray}
We will focus on the quantity in the brackets above.
Clearly,
\begin{equation}
\frac{pu'(\xi_n)}{\sqrt{r(t)}} \geq \frac{n \sqrt{r( \xi_n)}}{\sqrt{r(t)}} \geq n,
\end{equation}
since $r$ is increasing. Moreover, we also have that
\begin{eqnarray}
\int_{\xi_n}^t \left(q(s) - \lambda r(s) \right) u(s) \, ds & \geq & - \int_{\xi_n}^t \sqrt{r(s)} \, \left( \frac{q(s)}{r(s)} \, - \, \lambda \right)_- \, \sqrt{r(s)}u(s) \, ds \nonumber \\
& \geq & - C(u) \, \sqrt{r(t)} \, \int_{\xi_n}^t \left( \frac{q(s)}{r(s)} \, - \, \lambda \right)_-  \, ds \nonumber \\
& \geq & - C \, \sqrt{r(t)} \, ( x \, - \, \xi_n \, + \, 1), 
\end{eqnarray}
for all $t \leq x$. Here we have used that the negative part of $q/r$ is locally, uniformly integrable.
Thus,
\begin{equation}
 \frac{pu'(\xi_n) \, + \, \int_{\xi_n}^t \left(q(s) - \lambda r(s) \right) u(s) \, ds}{\sqrt{r(t)}}  \geq n - C(x - \xi_n +1).
\end{equation}
For $n$ sufficiently large, $(n > 2C)$, the choice $x = \xi_n + \alpha n$ with $0 \leq \alpha \leq 1/(2C)$ guarantees that
$[ \xi_n, x] \subset S_+$. By taking $2 C \alpha = 1$, we have demonstrated that
\begin{equation}
u(x) - u(\xi_n)  \geq  \left( \frac{n}{2} - C \right) \, \int_{\xi_n}^x \frac{r(t)}{p(t)} \, \frac{1}{\sqrt{r(t)}} \, dt,
\end{equation} 
and therefore the bound
\begin{equation}
\sqrt{r(x)}u(x) \geq \left( \frac{n}{2} - C \right) \int_{\xi_n}^x \frac{r(t)}{p(t)} \, dt,
\end{equation}
which contradicts the boundedness assumption on $\sqrt{r}u$. We have proven Proposition~\ref{prop:bdubdu'}.
\end{proof}

\begin{proof}(of Proposition~\ref{prop:l2bd})
Let $u$ be a solution of $\tau u =  \lambda u$ for which  (\ref{eq:uu'bd}) holds, 
and fix some $\alpha$ satisfying
\begin{equation} \label{eq:abd}
0 \, < \, \alpha \, < \,  \frac{I_-}{I_- \, + \, \sqrt{\gamma}}  \, \leq \, 1.
\end{equation} 
We claim that for any interval $[x-1,x+1] \subset (a,b)$,
there exists an $x_0 =x_0( \alpha) \in [x-1,x+1]$ for which
\begin{equation} \label{eq:unsmall}
\sqrt{r(x_0)} \, |u(x_0)| \, \geq \, \alpha \, C_1.
\end{equation}
Suppose that this is not the case. Then, let $[x-1,x+1] \subset (a,b)$ denote an interval for which
\begin{equation} \label{eq:usmall}
\sqrt{r(t)} |u(t)| \, < \, \alpha C_1,
\end{equation}
for all $t \in [x-1,x+1]$. Inserting (\ref{eq:usmall}) into (\ref{eq:uu'bd}), we find that
\begin{equation}
\frac{|pu'(t)|}{ \sqrt{r(t)}} \, \geq \, C_1 \, - \, \sqrt{r(t)} \, |u(t)| \, > \, (1- \alpha) \, C_1 \, > \, 0, 
\end{equation}
i.e., the derivative is strictly signed. Clearly then,
\begin{eqnarray}
\frac{ \alpha C_1}{ \sqrt{r(x+1)}} \, + \, \frac{ \alpha C_1}{ \sqrt{r(x-1)}} & > & | u(x+1) \, - \, u(x-1) | \nonumber \\
& = & \left| \int_{x-1}^{x+1} \frac{ r(t)}{p(t)} \, \frac{pu'(t)}{ \sqrt{r(t)}} \, \frac{1}{\sqrt{r(t)}} \, dt \, \right| \nonumber \\
& = &  \int_{x-1}^{x+1} \frac{r(t)}{p(t)} \, \frac{|pu'(t)|}{ \sqrt{r(t)}} \, \frac{1}{\sqrt{r(t)}} \, dt  \nonumber \\
& \geq &  C_1\, (1 - \alpha) \, \frac{1}{ \sqrt{r_+(x)}} \, \int_{x-1}^{x+1} \frac{r(t)}{p(t)} \, dt. 
\end{eqnarray}
This bound implies that
\begin{equation}
\frac{2 \alpha C_1}{ \sqrt{r_-(x)}} \, > \, C_1 (1 - \alpha) \frac{1}{\sqrt{r_+(x)}} 2 I_-, 
\end{equation}
and a short calculation reveals that this contradicts the range of $\alpha$ assumed in (\ref{eq:abd}).
We have proven (\ref{eq:unsmall}).

Additionally, for any $t \in [x-1,x+1]$ we may estimate,
\begin{eqnarray}
| u(t) \, - \, u(x_0) | & = & \left| \int_{x_0}^t \frac{r(s)}{p(s)} \, \frac{pu'(s)}{ \sqrt{r(s)}} \, \frac{1}{\sqrt{r(s)}} \, ds \right| \nonumber \\
& \leq & C_2 \, \frac{1}{ \sqrt{r_-(x)}} \, \int_{ \min(x_0,t)}^{\max(x_0,t)} \frac{r(s)}{p(s)} \, ds,
\end{eqnarray}
from which it is clear that
\begin{equation}
| \sqrt{r(t)}u(t) \, - \, \sqrt{r(t)}u(x_0) | \, \leq \, C_2 \, \sqrt{ \frac{r_+(x)}{r_-(x)}} \, \int_{\min(x_0,t)}^{\max(x_0,t)} \frac{r(s)}{p(s)} \, ds.
\end{equation}
Thus, given any $t \in [x-1,x+1]$ for which
\begin{equation}
\int_{\min(x_0,t)}^{\max(x_0,t)} \frac{r(s)}{p(s)} \, ds \, \leq  \, \frac{ \alpha}{2} \, \frac{r_-(x)}{r_+(x)} \, \frac{C_1}{C_2},
\end{equation}
it then follows that
\begin{eqnarray}
\sqrt{r(t)} u(t) & = & \sqrt{r(t)} u(x_0) + \sqrt{r(t)}u(t) - \sqrt{r(t)}u(x_0) \nonumber \\
& \geq & \sqrt{\frac{r(t)}{r(x_0)}} \sqrt{r(x_0)} u(x_0) - \left| \sqrt{r(t)}u(t) - \sqrt{r(t)}u(x_0) \right| \nonumber \\
& \geq & \sqrt{\frac{r_-(x)}{r_+(x)}} \alpha C_1 - C_2 \sqrt{\frac{r_+(x)}{r_-(x)}} \int_{\min(x_0,t)}^{\max(x_0,t)} \frac{r(s)}{p(s)} \, ds \nonumber \\
& \geq & \frac{\alpha}{2} C_1.
\end{eqnarray}
{F}rom this bound, the estimate (\ref{eq:l2bd}) easily follows, and we have proven Proposition~\ref{prop:l2bd}.
\end{proof}

\section{A Weidmann-type result for Sturm--Liouville operators}

As a first application we show how to obtain a generalization of a well-known result from
Weidmann \cite{wd} to the case of Sturm--Liouville operators. In fact, this is a simple generalization,
to the context of Sturm--Liouville equations, of \cite[Theorem 6]{Stolz2}.
For this theorem, we assume that the coefficients of $\tau$ are asymptotically Schr\"odinger like.
Specifically, we take $(a,b) = (0, \infty)$ and, in addition to the general assumptions provided in
(\ref{eq:baseass}), we assume that there exists a constant $c>0$ for which
\begin{equation} \label{eq:asysch}
1 - \frac{1}{p} \in L^1((c, \infty)) \quad {\rm and} \quad 1 - r \in L^1((c, \infty)).
\end{equation} 
Moreover, we assume that the potential $q= q_1 + q_2$ with
\begin{equation} \label{eq:qdec}
\begin{split}
& \hspace{2.5cm} q_1 \in L^1((c, \infty)) \\
& q_2 \mbox{ is } AC, \quad  q_2' \in L^1((c, \infty)), \quad {\rm and} \quad q_2(t) \to 0 \text{ as } t \to \infty.
\end{split} 
\end{equation} 
In this case, the following result holds.

\begin{lemma}
Let the coefficients of $\tau$ satisfy the general assumptions (\ref{eq:baseass}) and
also both (\ref{eq:asysch}) and (\ref{eq:qdec}). Then, for any solutions of $\tau u =  \lambda u$,
with $\lambda >0$, both $u$ and $pu'$ are bounded on $(c, \infty)$. 
\end{lemma}

\begin{proof}
It is sufficient to prove this result for real solutions $u \neq 0$. For any such solution, consider the function
\begin{equation}
h(t) : = \left( \lambda - q_2(t) \right) \, u(t)^2 + pu'(t)^2. 
\end{equation}
Given $0< \eps < \lambda$, there exists $t_0 \in (c, \infty)$ for which $\lambda - \eps \leq \lambda - q_2(t) \leq \lambda + \eps$ for
all $t \in [t_0, \infty)$. 
A short calculation shows that
\begin{equation}
h'(t) = -q_2'(t)u(t)^2 + \left[ q_1(t) + \left( q_2(t) - \lambda \right) \left( 1 - \frac{1}{p(t)} \right) + \lambda(1-r(t)) \right] 2 u(t) pu'(t),
\end{equation}
and therefore
\begin{equation}
|h'| \, \leq \, \left[ \frac{|q_2'|}{ \lambda - \eps} \, + \, \frac{|q_1|}{ \sqrt{ \lambda - \eps}} \, + \, \frac{1}{\sqrt{ \lambda - \eps}} \left( (\lambda + \eps) \left| 1 - \frac{1}{p} \right| \, + \, \lambda |1-r| \, \right) \right] h,
\end{equation}
for all $t \in [t_0, \infty)$. Thus the derivative of $\ln(h)$ is in $L^1((t_0, \infty))$, and hence, 
$\ln(h)$ has a finite limit at $\infty$. The same is then true for $h$, and thus, any such 
$u$ is bounded.
\end{proof}

As a consequence

\begin{theorem}
Let the coefficients of $\tau$ satisfy the general assumptions of (\ref{eq:baseass}). Moreover,
suppose that each of (\ref{eq:qass}), (\ref{eq:I-}), and (\ref{eq:gamfine}) hold.  

Then any self-adjoint extension $H$ of $\tau$ satisfies
\be
\sig_{ess}(H) = \sig_{ac}(H) = [0,\infty), \quad
\sig_{sc}(H) =\emptyset, \quad
\sig_{pp}(H) \subset (-\infty,0]
\ee
\end{theorem}

\begin{proof}
The claim about the essential spectrum follows from \cite[Thm.~15.2]{wdln}. By the previous
lemma the spectrum is purely absolutely continuous on $(0,\infty)$ and hence the result follows.
The condition that $r$ is non-decreasing is not needed since the conclusion of
Proposition~\ref{prop:bdubdu'} is already part of the previous lemma.
\end{proof}

%
%
%

\section{Radially Symmetric Quantum Trees}

Let $\Gamma$ be a rooted metric tree associated with the branching numbers $b_n$ at the
$n$'th level and distances $t_n$ of the vertices at the $n$'th level. We refer to \cite{so1}
for further details. We will set $t_0=0$ and assume, without loss of generality, that $\Gamma$
is regular in the sense of \cite{so1}, that is, $b_0 = 1$ and $b_k \geq 2$ for $k \geq 1$.
Furthermore, we will assume that the height
\begin{equation}
h_\Gamma = \lim_{n \to \infty} t_n = \infty
\end{equation}
since otherwise the spectrum of the Laplacian is purely discrete by \cite[Thm~4.1]{so1}.
The branching function, $g_{\Gamma}(t)$ is defined by
\begin{equation}
g_\Gamma(t) =  \prod_{n: t_n < t} b_n.
\end{equation}
Associated with $\Gamma$ is the Laplacian $-\Delta$ with a Dirichlet
boundary condition at the root and Kirchhoff boundary conditions at the vertices. We will consider
$-\Delta+V$, where $V$ is a radially symmetric potential depending only on the distance from
the root.

Then it is shown in \cite{car2,so1} that the study of $-\Delta+V$ can be reduced to the
Sturm--Liouville operators
\be
A = A_0 + V, \qquad
A_0 = \frac{1}{g_{\Gamma}}\left(-\frac{d}{dx}g_{\Gamma}\frac{d}{dx}\right), 
\quad x\in(0,\infty),
\ee
with a Dirichlet boundary condition at $x=0$. Here we think of $A_0$ as the Friedrich's extension
when restricted to functions with compact support and $V$ as some relatively form bounded
potential such that the operator sum is declared as a form sum.

\begin{theorem}[\cite{so1}]
Let $\Gamma$ be a metric tree generated by the sequences $\left\{t_n\right\}$ and $\left\{b_n\right\}$, then
\begin{equation}
-\Delta +V  \sim A \oplus \bigoplus_{k=1}^{\infty} (A|_{(t_k,\infty)})^{[b_0 \cdots b_{k-1}\cdot(b_k -1)]}.
\end{equation}
Here $A|_{(t_k,\infty)}$ denotes the restriction of $A$ to the interval $(t_k,\infty)$ with a
Dirichlet boundary condition at $x=t_k$ and $A^{[r]}$ denotes the orthogonal sum of $r$ copies
of the self-adjoint operator $A$ and $\sim$ denotes unitary equivalence.
\end{theorem}

Since we have $\sig_{ess}(A|_{(t_k,\infty)})=\sig_{ess}(A)$ and $\sig_{ac}(A|_{(t_k,\infty)})=\sig_{ac}(A)$
we can restrict our attention to $A$. As a second application of our results we note

\begin{theorem}\label{thm:qt}
Let $\Gamma$ be a metric tree and $-\Delta +V$ as before. Suppose $\inf_n (t_{n+1}-t_n)>0$ and
$\sup_n b_n <\infty$. Then
\be
\sig_{ac}(A) = \ol{\{ \lam\in\R | \sqrt{g_\Gamma} u \mbox{ is bounded for all solutions of } A u = \lam u\}}^{ess}
\ee
\end{theorem}

Next we want to consider the homogenous tree $\Gamma_0$, given by the following sequences:
\begin{equation}
t_n =  c\,n \quad\mbox{and}\quad b_n = b.
\end{equation}
In this case
\begin{equation}
g_{\Gamma_0}(t) = b^{\floor{t/c}}
\end{equation}
and all $A_0 |_{(t_k,\infty)}$ are unitarily equivalent to $A_0$.

\begin{lemma}[\cite{car1}]
For $\Gamma_0$ the spectrum of $A_0$ is given by
\begin{align}\nn
\sigma_{ac}(A_0) &= \bigcup_{l\in\N} \left[\left(\frac{\pi(l-1) + \theta}{c}\right)^2, \left(\frac{\pi l - \theta}{c}\right)^2\right],\\
\sigma_{sc}(A_0) & = \emptyset,\\ \nn
\sigma_{pp}(A_0) &= \{ \Big(\frac{\pi l}{c}\Big)^2 | l \in \N \}.
\end{align}
where
$$
\theta = \arccos\left(\frac{2}{b^{1/2} + b^{-1/2}}\right).
$$
\end{lemma}

\begin{proof}
For the sake of completeness and to introduce some items to be used later we provide the
elementary proof. We assume $c=1$ for notational simplicity.

The solutions of the differential equation $A_0 u = z$ satisfy $-u''(x)= z u(x)$ for $x\not\in\N$
and at $x=n\in\N$ we have the matching conditions
$$
u(n-) = u(n+) \mbox{ and } u'(n-) =bu'(n+), \qquad n \in \N.
$$
Thus the transfer matrix of $A_0 u = z$ is given by
\begin{align*}
T_0(z,x,0) &= \begin{pmatrix}
\cos(\sqrt{z}y) & -\sqrt{z} \sin(\sqrt{z}y) \\
\frac{1}{\sqrt{z}} \sin(\sqrt{z}y) & \cos(\sqrt{z}y)
\end{pmatrix} M(z)^n, \quad n=\floor{x},\: y= x-n,\\
M(z) &= \begin{pmatrix}
\frac{1}{b}\cos(\sqrt{z}) & -\frac{\sqrt{z}}{b} \sin(\sqrt{z}) \\
\frac{1}{\sqrt{z}} \sin(\sqrt{z}) & \cos(\sqrt{z})
\end{pmatrix}.
\end{align*}
In particular there are two solutions
\be\label{eqsola0}
u_{0,\pm}(z,x) = \ti{u}_{0,\pm}(z,x) \frac{\E^{\pm \alpha(z) x}}{\sqrt{g_{\Gamma_0}(x)}},
\ee
where $\ti{u}_\pm(z,x)$ is bounded and
$$
\alpha(z) = \log\left( \frac{1+b}{2\sqrt{b}} \cos(\sqrt{z}) + \sqrt{\frac{(1+b)^2}{4b}  \cos^2(\sqrt{z})-1}\right)
$$
with branch of the root chosen such that $\re(\alpha)\ge 0$. Hence the absolutely continuous spectrum
is given by $\sig_{ac} = \{\lam\in\R| \re(\alpha(\lam))=0\} = \{\lam\in\R| \frac{(1+b)^2}{4b}  \cos^2(\sqrt{\lam})\le 1\}$.
\end{proof}

Note that the unitary operator $U: L^2((0,\infty), g_{\Gamma_0}dx) \to L^2(0,\infty)$ given by
$u(x) = \sqrt{g_{\Gamma_0}(x)} u(x)$ maps $A_0$ to a Schr\"odinger operator
with a periodic $\delta'$ interaction. Hence the appearance of the band structure and
the similarity to periodic operators is no coincidence.

\begin{corollary}
Suppose
\be
V \in L^1(0,\infty),
\ee
then the essential spectrum of $A=A_0+V$ is given by
\be
\sigma_{ac}(A) = \sigma_{ac}(A_0)
\ee
and the essential spectrum is purely absolutely continuous in the interior. In particular
$\sig_{sc}(A)=\emptyset$.
\end{corollary}

\begin{proof}
Using $V \in L^1(0,\infty)$ one can use standard techniques (derive an integral equation using
variation of constants and solve it using the contraction principle) to show that the equation
$A u = z$ for $z\in\C$ away from the band edges has solutions which asymptotically look like
the solutions $u_{0,\pm}(z,x)$ of $A_0 u = z$ given in (\ref{eqsola0}). Hence the result follows from
Theorem~\ref{thm:qt}.
\end{proof}

Note that the results from \cite{kt} apply in this situation to determine when
a perturbation introduces a finite, respectively, infinite number of eigenvalues into
the spectral gaps.

\subsection*{Acknowledgments.}
We are indebted to the referee for constructive remarks.

\end{document}